\documentclass[12pt, reqno]{amsart} 
\usepackage{amsmath, amssymb, amsthm} 
\usepackage{geometry} 
\usepackage{enumitem} 
\usepackage{hyperref} 
\hypersetup{
    colorlinks=true,
    linkcolor=blue,
    citecolor=blue,
    urlcolor=blue
}
\usepackage[backend=biber, style=numeric, sorting=nyt]{biblatex} 

\usepackage{tikz}
\usetikzlibrary{arrows.meta}

\theoremstyle{plain} 
\newtheorem{theorem}{Theorem}[section]
\newtheorem{lemma}[theorem]{Lemma}
\newtheorem{proposition}[theorem]{Proposition}

\theoremstyle{definition} 
\newtheorem{definition}[theorem]{Definition}

\newtheorem{assumption}[theorem]{Assumption}

\theoremstyle{remark} 
\newtheorem{remark}[theorem]{Remark}

\begin{document}

\title[SHORT TITLE FOR HEADER]{Conformal product structures on $S^3$}

\author{Xianfeng JIANG}
\address{Laboratoire mathématique d'Orsay, Université Paris-Saclay, 91405 Orsay, France}
\email{xianfeng.jiang@universite-paris-saclay.fr}



\begin{abstract}
 We construct a conformal product structure on $S^3$ using the Reeb foliation. This is the first example of a conformal product structure on a compact simply connected manifold which is not built on a global product manifold.
\end{abstract}

\maketitle
\pagestyle{plain}

\section{Introduction}

On every Riemannian manifold $(M,g)$, there is a unique torsion-free linear connection $\nabla$ compatible with the Riemannian metric, called the Levi-Civita connection. More generally, on a conformal manifold $(M,c)$, where $c$ is a conformal class of Riemannian metrics, a torsion-free linear connection $D$ compatible with the conformal structure $c$ is called a Weyl connection \cite{weylraumzeitmaterie}. With respect to a metric $g\in c$, the Weyl connections are in one-to-one correspondence with 1-forms on $M$ (see \cite{gauduchon1995weyleinstein}). The corresponding 1-form is called the
Lee form of the Weyl connection with respect to $g$.

A Weyl connection is called closed or exact if its Lee form with respect to some metric in $c$ is closed or exact. This definition does not depend on the choice of the metric in the conformal class and it is easy to see that a Weyl connection is exact (or closed) if and only if it is globally (or locally) the Levi-Civita connection of a metric in $c$.

A Weyl connection $D$ on a smooth manifold $M$ is called reducible if its holonomy representation is decomposable, or, equivalently, if the tangent bundle decomposes into a direct sum of two non-trivial $D$-parallel subbundles $TM=T_1\oplus T_2$. In conformal geometry, a conformal product structure on a Riemannian manifold $(M, g)$ is defined by a Weyl connection that preserves a pair of orthogonal, complementary distributions. 

The importance of conformal product structures stems from the Merkulov-Schwachh{\"o}ofer classification \cite{merkulovschwachhofer1999} of possible holonomy groups of torsion-free connections. Indeed, when restricted to the special case of Weyl connections, the main result in \cite{merkulovschwachhofer1999} has the following consequence: if $D$ is a non-closed Weyl connection on a conformal manifold $(M,c)$ of dimension different from $4$, then either the holonomy of $D$ is the whole $CO(n)$, or $D$ is reducible (whence defines a conformal product structure).

A reducible Weyl connection which is closed but not exact is called a locally conformal product structure (LCP) \cite{flamencourt2024lcp}. The study of closed Weyl structures on a conformal manifold $M$ can be better understood in terms of the universal cover $\tilde{M}$. Indeed, the lift of a closed Weyl structure $D$ to $\tilde{M}$ is exact, meaning that it is the Levi-Civita connection of a Riemannian metric $h_D$ on $\tilde{M}$. A global structure theorem obtained by Matveev and Nikolayevsky \cite{matveevnikolayevsky2017} in the analytic case and by Kourganoff \cite{kourganoff2019similarity} in the general, smooth case, says that if a closed, non-exact Weyl structure $D$ on a compact conformal manifold $(M,c)$ is non-flat and has reducible holonomy, then the Riemannian manifold $(\tilde{M},h_D)$ is isometric to the Riemannian product $\mathbb{R}^q\times(N,g_N)$ where $\mathbb{R}^q$ (the flat part) is a Euclidean space and $(N,g_N)$ (the non-flat part) is an irreducible, non-complete manifold. 

For the non-closed case, the only known examples of conformal product structures are constructed on products $M=M_1\times M_2$ of two smooth manifolds. They are defined by a Riemannian metric of the form
\[
 g=e^{2f_1}g_1+e^{2f_2}g_2,
\]
where $g_i$ are Riemannian metrics on $M_i$ and $f_i$ are smooth functions on $M_1\times M_2$ for $i\in\{1,2\}$, and by a reducible Weyl connection $D$, with Lee form $-\text{d}^{M_1}f_2-\text{d}^{M_2}f_1$ with respect to $g$. In this case, $(M,[g],D)$ is called a global conformal product manifold. Conversely, every manifold with conformal product structure can be written locally in this way (c.f. Prop 2.3 of \cite{MoroianuPilca2025}). So it is natural to ask whether there is any conformal product structure on a compact simply connected manifold which is not a global product. The present article is devoted to answering this question positively. Prior to our work, (non-closed) conformal product structures have been studied on Riemannian manifolds with
special holonomy: in \cite{belgunflamencourtmoroianu2025weyl}, \cite{MoroianuPilca2025} for the cases of Kähler manifolds, in \cite{moroianupilca2024einstein} and \cite{conformaleinstein} for the
cases of Einstein manifolds, in \cite{moroianupilca2025reducible} for the case of reducible Riemannian manifolds and in \cite{jiang2026} for the case of constant sectional curvature manifolds.

The Reeb foliation, introduced by Georges Reeb in 1952 \cite{reeb1952sur},  is a foliation of codimension 1 on $S^3$ whose leaves are all diffeomorphic to $\mathbb{R}^2$ except for a compact leaf which is a torus. The torus leaf divides $S^3$  into two solid tori—known as Reeb components—glued along a common compact boundary leaf isomorphic to a torus $T^2$. 

In this paper, we explicitly construct a conformal product structure on $S^{3}$ whose associated codimension-1 foliation is precisely the Reeb foliation. Geometrically, the existence of a conformal product structure is strictly equivalent to the condition that the integral manifolds of both distributions are totally umbilical. Since 1-dimensional foliations are trivially totally umbilical, constructing a conformal product structure associated with the Reeb foliation reduces to finding a global smooth metric on $S^3$ under which the Reeb leaves are totally umbilical. 

Our construction strategy proceeds as follows: First, we endow the interior of each solid torus with a rotationally symmetric Riemannian metric. By enforcing the totally umbilical condition on the leaves, we derive the governing differential equations for the metric components. The primary technical difficulty lies in ensuring the smooth extension of this metric across the geometric singularities: the core circles of the solid tori and the compact boundary torus $T^2$. We find the sufficient conditions (for the functions used to define the foliation and the metric) which can ensure the smoothness and umbilical property of the metric and Reeb foliation on $S^3$. Then finally, we give a concrete construction of the Reeb foliation and a metric satisfying all the conditions we propose.


The paper is organized as follows: Section 2 introduces the explicit coordinate parameterization of the Reeb foliation on $S^3$ and derives the equivalent geometric conditions for conformal product structures, i.e., all the leaves of $\mathcal{F}_1$ and $\mathcal{F}_2$ are umbilical. In section 3 we recall some useful classical lemmas and establish the rigorous $C^\infty$ smoothness conditions for the Reeb foliation across the compact leaf $T^2$. Section 4 explicitly derives the governing differential equations for the metric components which are equivalent to the umbilical property on the single solid torus $(D^2 \times S^1,g)$ and finds the sufficient conditions to resolve the metric singularity at the core circle. Section 5 finds sufficient conditions which can ensure that the metric on the interiors of Reeb components $\mathring{V}_1$ and $\mathring{V}_2$ (which are isomorphic to $(D^2\times S^1,g)$) can be extended to a smooth metric $\tilde{g}$ on $S^3$ and establishes that the compact leaf $T^2$ is also umbilical with respect to $\tilde{g}$. Section 6 gives a precise construction of Reeb foliation and the smooth metric $g$ on $S^3$ and establishes that all the conditions we need are satisfied.

\subsection*{Acknowledgments} The author would like to thank Andrei Moroianu for his guidance and support. This work was financially supported by a public doctoral contract from Université Paris-Saclay, via the École Doctorale de Mathématiques Hadamard (EDMH).

\section{Preliminaries}

\subsection{Reeb Foliation}

We begin with an explicit description of the Reeb foliation. We consider $S^3$ defined as a closed submanifold of $\mathbb{R}^4$ by:
\[
S^3=\{(x_1,x_2,x_3,x_4)\in\mathbb{R}^4\ |\ x_1^2+x_2^2+x_3^2+x_4^2=1\}.
\]
We denote $z_1=x_1+ix_2$ and $z_2=x_3+ix_4$, then $S^3=\{(z_1,z_2)\in\mathbb{C}^2\ |\ |z_1|^2+|z_2|^2=1\}$. It can be decomposed into two solid tori $V_1$ and $V_2$, which are called Reeb components:
\[
V_1=\{(z_1,z_2)\in S^3\ |\ |z_1|^2\ge\frac12\},\quad V_2=\{(z_1,z_2)\in S^3\ |\ |z_2|^2\ge\frac12\}.
\]
The two components $V_1$ and $V_2$ are closed subsets of $S^3$, their boundaries are 
\[
\partial V_1=\partial V_2=\{(z_1,z_2)\in S^3\ |\ |z_1|^2=|z_2|^2=\frac12\}\simeq S^1\times S^1.
\]
We denote the common boundary by $T^2$, and the central slice of $V_i$ is $C_i:=\{(z_1,z_2)\in S^3\ |\ |z_i|=1\}$ for $i=1,2$, which are two embedded submanifolds in $S^3$ diffeomorphic to the circle $S^1$. We denote the interior of $V_1$, $V_2$ by $\mathring{V_1}$ and $\mathring{V_2}$ respectively.

The interiors of the Reeb components are diffeomorphic to the solid torus $\mathring{D}^2\times S^1$.  We consider two maps: 
\begin{equation*}
\begin{aligned}
\Psi_1:\;
\begin{array}{rcl}
\mathring{V_1} & \longrightarrow & \mathring{D}^2\times S^1 \\
(z_1,z_2) & \longmapsto &
\left(\sqrt2 z_2,\dfrac{z_1}{|z_1|}\right)
\end{array}
\qquad
\Psi_2:\;
\begin{array}{rcl}
\mathring{V_2} & \longrightarrow & \mathring{D}^2\times S^1 \\
(z_1,z_2) & \longmapsto &
\left(\sqrt2 z_1,\dfrac{z_2}{|z_2|}\right)
\end{array}
\end{aligned}
\end{equation*}
 One can establish that $\Psi_1$ and $\Psi_2$ are diffeomorphisms, their inverse maps are: 
\begin{equation*}
\begin{aligned}
\Phi_1:\;
\begin{array}{rcl}
\mathring{D}^2\times S^1 & \longrightarrow & \mathring{V}_1 \\
(w_1,w_2) & \longmapsto &
\left(
w_2\sqrt{1-\dfrac{|w_1|^2}{2}},
\dfrac{w_1}{\sqrt2}
\right)
\end{array}
\end{aligned}
\end{equation*}
\begin{equation*}
\begin{aligned}
\Phi_2:\;
\begin{array}{rcl}
\mathring{D}^2\times S^1 & \longrightarrow & \mathring{V}_2 \\
(w_1,w_2) & \longmapsto &
\left(
\dfrac{w_1}{\sqrt2},
w_2\sqrt{1-\dfrac{|w_1|^2}{2}}
\right)
\end{array}
\end{aligned}
\end{equation*}
which induce diffeomorphisms between $C_i$ and the central circle $C:=\{0\}\times S^1$ of the solid torus $\mathring{D}^2\times S^1$. Hence both $\mathring{V_1}$ and $\mathring{V_2}$ are diffeomorphic to a solid torus $\mathring{D^2}\times S^1$. The solid torus $D^2\times S^1$ has torus coordinates $(r,\phi,\gamma)$, where $(r,\phi)$ are the polar coordinates in the disk $D^2\setminus\{0\}$ and $\gamma$ is the unit complex number parametrization of $S^1$ with $0<r<1$, $\phi\in[-\pi,\pi)$, and $\gamma\in S^1=\{z\in\mathbb{C}\ |\ |z|=1\}$. So $\mathring{V_1}$ and $\mathring{V_2}$ have coordinates induced by the torus coordinate $(r,\phi,\gamma)$ and the diffeomorphisms $\Psi_1$ and $\Psi_2$. We denote them by $(r_1,\phi_1,\gamma_1)$ and $(r_2,\phi_2,\gamma_2)$ respectively:
\begin{equation}\label{coordinates}
\begin{aligned}
(r_1,\phi_1,\gamma_1):&=(r\circ\Psi_1,\phi\circ\Psi_1,\gamma_1\circ\Psi_1)=(\sqrt2|z_2|,\text{Arg}(z_2), \frac{z_1}{|z_1|}),\quad \forall(z_1,z_2)\in V_1,\\
(r_2,\phi_2,\gamma_2):&=(r\circ\Psi_2,\phi\circ\Psi_2,\gamma_2\circ\Psi_2)=(\sqrt2|z_1|,\text{Arg}(z_1),\frac{z_2}{|z_2|}),\quad \forall(z_1,z_2)\in V_2.
\end{aligned}
\end{equation}
Hence $r_i$ are functions on $\mathring{V_i}$ with values in $[0,1)$, $\phi_i$ and $\gamma_i$ are smooth complex-valued functions from $\mathring{V_i}$ to $S^1\subset\mathbb{C}$.

The Reeb foliation is given by constructing foliations on the two solid tori separately and then gluing them. We consider the covering map of the solid torus: 
\[
\pi:D^2\times\mathbb{R}\rightarrow D^2\times S^1,\quad (r,\phi,t)\mapsto (r,\phi,e^{it}).
\]
Through this covering, $t \in \mathbb{R}/2\pi\mathbb{Z}$ serves as the local real angular coordinate on $S^1$ corresponding to the complex parameter $\gamma = e^{it}$. This naturally defines the smooth coordinate vector field $\partial_t$ and the differential 1-form $dt$ on the solid torus. Then we consider a smooth function $f:[0,1)\rightarrow\mathbb{R}$ such that the following assumption holds:
\begin{assumption}\label{assumption0}
    We assume that $f:[0,1)\rightarrow\mathbb{R}$ is a smooth function such that:\\
    i) $f'(r)>0$ for $0<r<1$,\\
    ii) $\lim_{r\rightarrow 1}f(r)=+\infty$.\\
\end{assumption}
 In the cylinder $D^2\times \mathbb{R}$, the  foliation structure defined by the level sets:
\begin{equation}\label{H}
F(r,\phi,t)=t-f(r)=c,
\end{equation}
where $c$ is an arbitrary constant, is invariant under translation in $\mathbb{R}$, so it can be projected onto the solid torus $\mathring{D}^2\times S^1$ through the covering map $\pi$. We denote it by $\mathcal{F}_f$. 
Then by gluing the two solid tori by a Dehn twist, we obtain a $S^3$ and the Reeb foliation is induced by the foliation structure on the solid tori. We denote by $\tilde{\mathcal{F}_f}$ this Reeb foliation (on $S^3$) determined by function $f$. 

\subsection{Conformal Product Structures}

Let $(M,g)$ be a Riemannian manifold which has a decomposition of the tangent bundle $\text{T}M=T_1\oplus T_2$ into two integrable distributions. We denote the foliations induced by $T_1$ and $T_2$ by $\mathcal{F}_1$ and $\mathcal{F}_2$.

\begin{definition}
    A conformal product structure on $(M, g)$ consists of a Weyl connection $D$ together with a decomposition of the tangent bundle of $M$ as $\text{T}M=T_1\oplus T_2$, where $T_1$ and $T_2$ are orthogonal $D$-parallel non-trivial distributions. The rank of a conformal product structure is defined to be the smallest of the ranks of the two distributions $T_1$ and $T_2$. A conformal product structure $D$ on $(M,g)$ is said to be orientable if the $D$-parallel distributions $T_1$ and $T_2$ are orientable. Up to a finite cover, every conformal product structure is orientable.
\end{definition}

We recall the notion of totally umbilical submanifolds and establish its relation with conformal product structures.

\begin{definition} For a submanifold $L$ in $(M,g)$, and vector fields $X,Y\in \text{T}L$, the second fundamental form is $h(X,Y):=(\nabla_XY)^{\bot}$. If there exists a normal vector field $H$ such that 
\[
h(X,Y)=g(X,Y)H,
\]
then the submanifold $L$ is called totally umbilical. An integrable distribution is called umbilical if and only if its induced submanifolds are all totally umbilical.
\end{definition}
\begin{proposition}\label{prop fond}
     There exists a Weyl connection $D$ preserving $T_1$ and $T_2$ if and only if the leaves of $\mathcal{F}_1$ and $\mathcal{F}_2$ are totally umbilical.
\end{proposition}

\begin{proof}
    For $X,Y\in \Gamma(T_1)$ and $Z\in \Gamma(T_2)$, we have
\begin{equation}\label{1}
\begin{aligned}
g(D_XY,Z)&=g(\nabla_XY+\theta(X)Y+\theta(Y)X-g(X,Y)\theta^\sharp,Z)\\
&=g(\nabla_XY,Z)-g(X,Y)\theta(Z)\\
&=g(h(X,Y)-g(X,Y)\theta^\sharp,Z)
\end{aligned}
\end{equation}

The Weyl connection $D$ preserves $T_1$ if and only if $g(D_XY,Z)=0$ for all $X,Y\in \Gamma(T_1)$ and $Z\in \Gamma(T_2)$, which by \eqref{1} is equivalent to $h(X,Y)=g(X,Y)\theta_2^\sharp$ i.e. to the fact that the foliation $\mathcal{F}_1$ is totally umbilical and the fundamental form is parallel with the projection of $\theta$ on $T_2$. Consequently, the Weyl connection $D$ preserves both $T_1$ and $T_2$ if and only if $\mathcal{F}_1$ and $\mathcal{F}_2$ are both totally umbilical.
\end{proof}

\section{Smoothness of Reeb foliations}

In order to establish the smoothness of $\mathcal{F}_f$, we need to recall a series of classical lemmas: 
\begin{lemma}\label{hadamard}
(Hadamard's lemma) If a function $H(x)\in\mathcal{C}^\infty(U)$ is defined on a convex neighborhood $U$ of $\ 0\in\mathbb{R}$, and $H(0)=0$, then there exists a smooth function $K(x)\in\mathcal{C}^\infty(U)$ such that $H(x)=x\cdot K(x)$.
\end{lemma}
\textbf{Remark:} The construction of $K$ is given by:
\[
K(x)=\int_0^1H'(tx)\text{d}t
\]
so $K(0)=\int_0^1H'(0)\text{d}t=H'(0)$.

We also need the following lemma:
\begin{lemma}\label{simple}
    If a smooth function $F:(-1,1)\rightarrow \mathbb{R}$ satisfies $F^{(n)}(0)=0$ for any $n\in\mathbb{N}$, then there exists a smooth function $G:(-1,1)\rightarrow\mathbb{R}$ such that  $F(r)=G(r^2)$, $\forall r\in[0,1)$. 
\end{lemma}
\begin{proof}
We define 
    \[
    G(y):=\begin{cases}
        F(\sqrt y),\ \ y\ge 0\\
        0,\qquad\ \  y<0.
    \end{cases}
    \]
 hence $F(r)=G(r^2),\ \forall r\in[0,1)$. We need to prove that $G$ is smooth at $0$ i.e., that $\lim_{y\rightarrow0^+}G^{(k)}(y)=0$.

Since $F^{(n)}(0)=0, \forall n\in\mathbb{N}$, we have $\lim_{x\rightarrow0^+}\frac{F^{(n)}(x)}{x^m}=0, \forall n,m \in\mathbb{N}$. It suffices to prove that $G^{(k)}$ is a finite linear combination (with coefficients in $\mathbb{R}$) of terms of the form: $\frac{F^{(j)}(\sqrt y)}{(\sqrt y)^{2k-j}}$ for $1\le j\le k$. This follows immediately by induction.
\end{proof}

\begin{lemma}\label{whitney}
    For a smooth function $F:(-1,1)\rightarrow\mathbb{R}$, the following are equivalent:\\
    i) $F^{(2k+1)}(0)=0,\forall  k\in\mathbb{N}$,\\
    ii) there exists a smooth function $G:(-1,1)\rightarrow\mathbb{R}$ such that  $F(x)=G(x^2)$, $\forall x\in[0,1)$.
\end{lemma}

\begin{proof}
     We assume $F^{(2k+1)}(0)=0,\forall  k\in\mathbb{N}$. By Borel's lemma (c.f. Theorem 3.3 of \cite{lafontaine}), there is a smooth function $L\in\mathcal{C}^\infty(\mathbb{R})$ that satisfies:
    \[
    L^{(k)}(0)=\frac{F^{(2k)}(0)}{(2k)!}k!.
    \]
    We define $\tilde{F}(x)=L(x^2)$, which is a smooth even function. Hence $R(x):=F(x)-\tilde{F}(x)$  satisfies $R^{(k)}(0)=0$, $\forall k\in\mathbb{N}$. 
    Then by Lemma \ref{simple}, there exists a smooth function $\tilde{G}:(-1,1)\rightarrow \mathbb{R}$ such that $R(x)=\tilde{G}(x^2)$ for every $x\in[0,1)$. Then $G(y):=L(y)+\tilde{G}(y)$ is smooth and satisfies $G(x^2)=L(x^2)+\tilde{G}(x^2)=\tilde{F}(x)+R(x)=F(x)$.

   The converse is straightforward: if ii) holds, $F(x)$ coincides with $G(x^2)$ on $[0,1)$. Since $G(x^2)$ is a smooth even function on $(-1, 1)$, all of its odd derivatives at $0$ vanish. Because $F$ is smooth on $(-1, 1)$, its derivatives at $0$ equal its right derivatives at $0$, which match those of $G(x^2)$. Thus $F^{(2k+1)}(0) = 0$ for all $k \in \mathbb{N}$.
\end{proof}

\begin{lemma}\label{lemma again}
    Let $F:D^2\rightarrow\mathbb{R}$ be defined by $F(x,y)=f(r)$ for a function $f:[0,1)\rightarrow\mathbb{R}$, where $r=\sqrt{x^2+y^2}$, i.e., $F(x,y)$ is determined by the distance from $(x,y)$ to $(0,0)$. Then $F$ is smooth at $(0,0)$ if and only if there exists a function $h\in\mathcal{C}^\infty((-1,1))$ such that $f(r)=h(r^2),\ \forall r\in[0,1)$. 
\end{lemma}
\begin{proof}
    We consider $D^2$ as the unit disc of $\mathbb{R}^2$, i.e., $D^2=\{(x,y)\in\mathbb{R}^2\ |\ x^2+y^2\le1\}$. The function $\tilde{f}:(-1,1)\rightarrow\mathbb{R}$ defined by:
    \[
    \tilde{f}(x):=\begin{cases}
        f(x),\quad x\ge 0\\
        f(-x),\ x<0
    \end{cases}
    \]
    is even and smooth, since $\tilde{f}(x)=F(x,0),\forall x\in(-1,1)$. Hence by Lemma \ref{whitney}, there exists a function $h\in\mathcal{C}^\infty((-1,1))$ such that $f(r)=h(r^2)$. 

    Conversely, if $f(r)=h(r^2)$ for $r\in [0,1)$, then $F(x,y)=f(r)=h(r^2)=h(x^2+y^2)$ which is the composition of smooth functions $r^2(x,y)=x^2+y^2$ and $h$, so $F$ is smooth.
\end{proof}

From now on we assume that $f:[0,1)\rightarrow\mathbb{R}$ defined in \eqref{H} is a smooth function such that $f'(r)>0$ for $0<r<1$ and $\lim_{r\rightarrow 1}f(r)=+\infty$.
\begin{proposition}\label{condition1}
    The foliation $\mathcal{F}_f$ on $\mathring{D}^2\times S^1$ is smooth at $C_1$ and $C_2$ if $f$ satisfies: $f^{(2k+1)}(0)=0, \forall k\in\mathbb{N}$.
\end{proposition}
\begin{proof}
     By definition, the Reeb foliation is smooth at $C_1$ and $C_2$ if the function $F(x,y):= f(\sqrt{x^2+y^2})$ is smooth at $(0,0)\in\mathbb{R}^2$. By Lemma \ref{whitney} and Lemma \ref{lemma again}, this is equivalent to $f^{(2k+1)}(0)=0,\forall k\in\mathbb{N}$.
\end{proof}

We now set up a chart construction of a neighborhood of the compact leaf $T^2$.

For any $\epsilon\in(0,1)$, we can define a cube in $\mathbb{R}^3$ and its corresponding open set, a neighborhood of $T^2$ denoted by $O$:
\begin{equation}\label{V,O}
\begin{aligned}
&W:=\{(\rho,\alpha,\beta)\in\mathbb{R}^3\ |\ \rho\in(-\epsilon,\epsilon),\alpha\in(-\pi,\pi),\beta\in(-\pi,\pi)\}.\\
&O:=\{(z_1,z_2)\in S^3\ |\ -\epsilon<|z_1|^2-|z_2|^2<\epsilon,z_1\notin\mathbb{R}_{\le 0},z_2\notin\mathbb{R}_{\le 0}\}.
\end{aligned}
\end{equation}
For $(z_1,z_2)\in O$, since $|z_1|^2+|z_2|^2=1$ and $-\epsilon<|z_1|^2-|z_2|^2<\epsilon$, we have:
\[
|z_1|^2>\frac{1-\epsilon}{2}>0,\quad |z_2|^2>\frac{1-\epsilon}{2}>0.
\]
We will now construct a diffeomorphism between $W$ and $O$:
We define a map $\Psi:O\rightarrow W$ as 
\[
\Psi(x_1,x_2,x_3,x_4)=(\rho,\alpha,\beta)
\]
    where $(\rho,\alpha,\beta)$ can be expressed respectively as: 
    \begin{equation}
        \begin{aligned}
          \rho&=|z_1|^2-|z_2|^2=x_1^2+x_2^2-x_3^2-x_4^2\\
        \alpha&=\text{Arg}(z_1)=\text{Arg}(x_1+ix_2)\\
        \beta&=\text{Arg}(z_2)=\text{Arg}(x_3+ix_4),
        \end{aligned}
    \end{equation}
    where $\alpha, \beta$ are chosen to be contained in $(-\pi,\pi)$. These functions are smooth by definition of $O$, hence $\Psi:O\rightarrow W$ is smooth.
    We construct its inverse $\Phi:W\rightarrow O$: $\Phi(\rho,\alpha,\beta)=(x_1,x_2,x_3,x_4)$ where
    \begin{equation}
        \begin{aligned}
            &x_1(\rho,\alpha,\beta)=\sqrt{\frac{1+\rho}{2}}\cos\alpha,\quad x_2(\rho,\alpha,\beta)=\sqrt{\frac{1+\rho}{2}}\sin\alpha\\
            &x_3(\rho,\alpha,\beta)=\sqrt{\frac{1-\rho}{2}}\cos\beta,\quad x_4(\rho,\alpha,\beta)=\sqrt\frac{1-\rho}{2}\sin\beta.
        \end{aligned}
    \end{equation}
In other words,
    \[z_1=\sqrt{\frac{1+\rho}{2}}e^{i\alpha},\quad z_2=\sqrt{\frac{1-\rho}{2}}e^{i\beta}
    \]
It is simple to establish that they are inverse mappings of each other and hence diffeomorphisms. We denote by $\Phi^*$ the pullback of differential forms through $\Phi$, and by $\Phi_*$ the pushforward of vector fields.
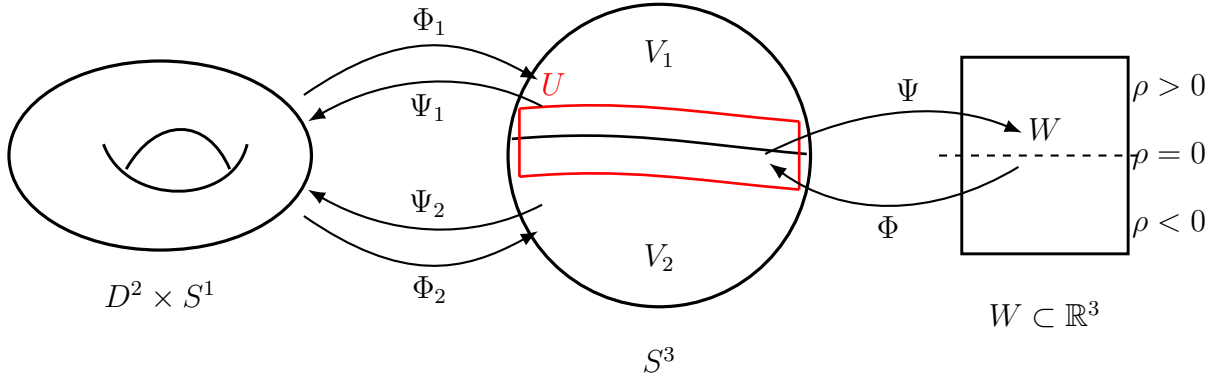
\begin{figure}[ht]
\centering
\begin{tikzpicture}[
    >=Latex,
    thick
]


\draw[line width=1.2pt]
    (0,0) ellipse (2.0 and 1.25);

\draw[line width=1.1pt]
    (-0.75,0.15)
    .. controls (-0.45,-0.68)
             and (0.95,-0.68) ..
    (1.15,0.15);

\draw[line width=1.1pt]
    (-0.45,-0.18)
    .. controls (-0.05,0.52)
             and (0.62,0.52) ..
    (0.92,-0.18);

\node at (0,-1.85)
    {$D^2\times S^1$};

\draw[->]
    (1.9,0.8)
    .. controls (3.2,1.7) and (4.0,1.55) ..
    node[above] {$\Phi_1$}
    (5.0,1.0);

\draw[->]
    (1.9,-0.8)
    .. controls (3.2,-1.7) and (4.0,-1.55) ..
    node[below] {$\Phi_2$}
    (5.0,-1.0);


\draw[line width=1.2pt]
    (6.6,0) circle (2.0);

\draw[line width=1pt]
    (4.65,0.22)
    to[out=5,in=175]
    (8.55,0.02);

\draw[red,line width=1pt]
    (4.75,0.62)
    to[out=5,in=175]
    (8.45,0.45);

\draw[red,line width=1pt]
    (4.75,-0.28)
    to[out=5,in=175]
    (8.45,-0.45);

\draw[red,line width=1pt]
    (4.75,0.62)--(4.75,-0.28);

\draw[red,line width=1pt]
    (8.45,0.45)--(8.45,-0.45);

\node[red] at (5.2,0.95)
    {$U$};

\node at (6.6,1.35)
    {$V_1$};

\node at (6.6,-1.35)
    {$V_2$};

\node at (6.6,-2.7)
    {$S^3$};


\draw[->]
    (8.0,0.02)
    .. controls (9.4,0.72) and (10.5,0.68) ..
    node[above] {$\Psi$}
    (11.4,0.3);



\draw[line width=1.1pt]
    (10.6,-1.3)
    rectangle
    (12.8,1.3);

\node at (11.7,0.35)
    {$W$};

\draw[dashed]
    (10.3,0)
    -- (13.0,0);

\node[right] at (12.7,0.9)
    {$\rho>0$};

\node[right] at (12.7,0)
    {$\rho=0$};

\node[right] at (12.7,-0.9)
    {$\rho<0$};

\node at (11.7,-2.1)
    {$W\subset\mathbb R^3$};


\draw[->]
    (5.05,0.65)
    .. controls (4.1,1.15) and (3.0,1.05) ..
    node[below] {$\Psi_1$}
    (1.95,0.45);

\draw[->]
    (5.05,-0.65)
    .. controls (4.1,-1.15) and (3.0,-1.05) ..
    node[above] {$\Psi_2$}
    (1.95,-0.45);


\draw[->]
    (11.35,-0.15)
    .. controls (10.2,-0.85) and (9.0,-0.8) ..
    node[below] {$\Phi$}
    (8.05,-0.10);

\end{tikzpicture}
\caption{Reeb decomposition and foliated chart}
\end{figure}

The transition function between the charts $\Psi_1:\mathring{V}_1\rightarrow \mathring{D}^2\times S^1$ and $\Psi:O\rightarrow W$ is:
\begin{equation}\label{Psi_1 Phi}
\Psi_1\circ\Phi(\rho,\alpha,\beta)=\Psi_1(\sqrt{\frac{1+\rho}{2}}e^{i\alpha},\sqrt{\frac{1-\rho}{2}}e^{i\beta})=(\sqrt{1-\rho}e^{i\beta},e^{i\alpha})
\end{equation}
The transition function between the charts $\Psi_2:\mathring{V}_2\rightarrow\mathring{D}^2\times S^1$ and $\Psi: O\rightarrow W$ is:
\begin{equation}\label{Psi_2 Phi}
\Psi_2\circ\Phi(\rho,\alpha,\beta)=\Psi_2(\sqrt{\frac{1+\rho}{2}}e^{i\alpha},\sqrt{\frac{1-\rho}{2}}e^{i\beta})=(\sqrt{1+\rho}e^{i\alpha},e^{i\beta}).
\end{equation}
Choose an arbitrary point $P=(\rho,\alpha,\beta)\in W\subset\mathbb{R}^3$. By \eqref{coordinates}, we have:
\begin{itemize}
    \item  $\Phi(P)\in V_1$ if and only if $\rho\ge 0$. In this case, $r_1(\Phi(P))=r(\Psi_1\circ\Phi(P))=\sqrt{2}|z_2| = \sqrt{2} \sqrt{\frac{1-\rho}{2}}=\sqrt{1-\rho}$, $\phi_1(\Phi(P))=\phi\circ(\Psi_1\circ\Phi(P))=\text{Arg}(z_2)=\beta$, $\gamma_1(\Phi(P))=\gamma(\Psi_1\circ\Phi(P))=\frac{z_1}{|z_1|}=e^{i\alpha}$.
    \item  $\Phi(P)\in V_2$ if and only if $\rho\le 0$. In this case, $r_2(\Phi(P))=r\circ\Psi_2\circ\Phi(P)=\sqrt{2}|z_1|=\sqrt{2} \sqrt{\frac{1+\rho}{2}}=\sqrt{1+\rho}$, $\phi_2(\Phi(P))=\phi(\Psi_2\circ\Phi(P))=\text{Arg}(z_1)=\alpha$, $\gamma_2(\Phi(P))=\gamma(\Psi_2\circ\Phi(P))=\frac{z_2}{|z_2|}=e^{i\beta}$. 
\end{itemize}

\begin{proposition}\label{condition2}
    The Reeb foliation is smooth at $T^2$ if  $\lim_{r\rightarrow1^-}\frac{\text{d}^k}{\text{d}r^k}(\frac{1}{f'(r)})=0,\ \  \forall k\in\mathbb{N}$.
\end{proposition}
\begin{proof}


 We set the moving frame of $\text{T}\mathcal{F}_f$ on $\mathring{D}^2\times S^1$ (which is diffeomorphic to $\mathring{V_1}$ and $\mathring{V_2}$ by $\Phi_1$, $\Phi_2$):
\begin{equation}\label{F_i}
    F_1=\partial_\phi,\quad
    F_2=\frac{1}{f'(r)}\partial_r+\partial_t,\quad
\end{equation}
Then we define the moving frame of $\text{T}\tilde{\mathcal{F}_f}$ on $S^3\setminus(C_1\cup C_2\cup T^2)$:
\begin{equation}
    \tilde{F}_1(z):=\begin{cases}
        \Phi_{1*}(F_1), z\in\mathring{V_1}\\
        \Phi_{2*}(F_2), z\in\mathring{V_2}
    \end{cases},\quad
    \tilde{F}_2:=\begin{cases}
        \Phi_{1*}(F_2), z\in\mathring{V_1}\\
        \Phi_{2*}(F_1),z\in\mathring{V_2}
    \end{cases}\ .\quad
\end{equation}
In order to prove the smoothness of $\mathcal{F}_f$, we only need to establish the smoothness of $\tilde{F}_1$ and $\tilde{F}_2$ at $T^2$. They are smooth if and only if their pull back to $W$ through $\Phi$ is smooth:
\begin{equation*}
    \Psi_*(\tilde{F}_1):=\begin{cases}
        (\Psi\circ\Phi_1)_*(F_1)=\partial_\beta, \rho>0\\
        (\Psi\circ\Phi_2)_*(F_2)=\frac{2\sqrt{1+\rho}}{f'(\sqrt{1+\rho})}\partial_\rho+\partial_\beta, \rho<0
    \end{cases}
\end{equation*}
\begin{equation*}
    \Psi_*(\tilde{F}_2):=\begin{cases}
        (\Psi\circ\Phi_1)_*(F_2)=-\frac{2\sqrt{1-\rho}}{f'(\sqrt{1-\rho})}\partial_\rho+\partial_\alpha,\rho>0,\\
        (\Psi\circ\Phi_2)_*(F_1)=\partial_\alpha,\rho<0.
    \end{cases}\quad
\end{equation*}
Since $f'(r) \to +\infty$ as $r \to 1^-$, the frame continuously extends to $\tilde{F}_1 = \partial_\beta, \tilde{F}_2=\partial_\alpha$, which  span the tangent bundle of $T^2$.
The compact leaf $T^2$ corresponds to the $\rho=0$ slice in $W\subset\mathbb{R}^3$. If the condition $\lim_{r\rightarrow1^-}\frac{\text{d}^k}{\text{d}r^k}(\frac{1}{f'(r)})=0$ for any $k\in\mathbb{N}$ holds, by a simple application to Leibniz rule, we obtain: $\lim_{\rho\rightarrow 0^\mp}\pm\frac{\text{d}^k}{\text{d}\rho^k}\frac{2\sqrt{1\pm\rho}}{f'(\sqrt{1\pm\rho})}=0, \forall k\in\mathbb{N}$. Hence the vector fields $\tilde{F}_1$ and $\tilde{F}_2$ are smooth at $T^2$.
\end{proof}
\begin{remark}
    Notice that the foliated neighborhood $O$ (and its coordinate cube $W$) does not cover the entire compact torus leaf $T^2$ due to the branch cuts of the argument functions ($z_1, z_2 \notin \mathbb{R}_{\le 0}$). However, verifying the smoothness and the totally umbilical (or geodesic) conditions strictly within the chart $W$ is mathematically sufficient to establish the global results on the entire $T^2$. This is guaranteed by the inherent toral symmetry of our construction: both the Reeb foliation $\tilde{\mathcal{F}_f}$ and the constructed Riemannian metric $\tilde{g}$ are point-wise invariant under the standard diagonal toral action $T^2 = S^1 \times S^1$ acting by phase rotations $(z_1, z_2) \mapsto (e^{i\psi_1}z_1, e^{i\psi_2}z_2)$. Since this action consists of global isometries and acts transitively on the boundary torus $T^2$, any point lying on the boundary cuts can be smoothly rotated into the interior of the chart $O$.
\end{remark}

\begin{assumption}\label{assumption1}
    We assume from now on that $f:[0,1)\rightarrow\mathbb{R}$ is a smooth function satisfying Assumption \ref{assumption0}, such that:\\
    iii) $f^{(2k+1)}(0)=0,\forall k\in\mathbb{N}$.\\
    iv) $\lim_{r\rightarrow1^-}\frac{\text{d}^k}{\text{d}r^k}(\frac{1}{f'(r)})=0,\ \  \forall k\in\mathbb{N}$.
\end{assumption}
\begin{remark}
    Conditions iii), iv) ensure of the smoothness of the foliation at $C_1$, $C_2$ and $T^2$ by Proposition \ref{condition1} and Proposition \ref{condition2} respectively.
\end{remark}
For example, we can construct such a function $f$ as:
\begin{equation}\label{f}
f(r)=\int_0^rt\text{exp}(\frac{t^2}{1-t^2})\text{d}t
\end{equation}
which satisfies all the conditions in Assumption \ref{assumption1}. 

\section{Construction of conformal product structure on $\mathring{D}^2\times S^1$}

The main goal of this section is to construct a conformal product structure on the solid torus $\mathring{D}^2\times S^1$. We consider a smooth function $f\in\mathcal{C}^\infty((0,1))$ and set the moving frame $\{E_i\}$ on $(\mathring{D}^2\setminus\{0\})\times S^1$ as follows:
\begin{equation}\label{E_i}
E_1=\partial_\phi,\quad E_2=\partial_r+f'(r)\partial_t,\quad E_3=-f'(r)\partial_r+\partial_t.
\end{equation}
Their dual basis $\{E_i^*\}$ can be expressed as follows:
\begin{equation}
\begin{aligned}
E_1^*&=\text{d}\phi,\\ E_2^*&=\frac{1}{1+f'(r)^2}\text{d}r+\frac{f'(r)}{1+f'(r)^2}\text{d}t,\\ 
E_3^*&=\frac{-f'(r)}{1+f'(r)^2}\text{d}r+\frac{1}{1+f'(r)^2}\text{d}t,
\end{aligned}
\end{equation}
which are well defined on $U:=(\mathring{D}^2\setminus\{0\})\times S^1$ and smooth. For any smooth positive functions $A(r),B(r),C(r)$ on $(0,1)$, we can define a metric $g_{f,A,B,C}$ on $U=(\mathring{D}^2\setminus\{0\})\times S^1$ by: 
\begin{equation}
\begin{aligned}
g(E_1,E_1)&=A(r)^2,\ \ g(E_2,E_2)=B(r)^2,\ \ g(E_3,E_3)=C(r)^2,\\
&g(E_1,E_2)=g(E_2,E_3)=g(E_1,E_3)=0. 
\end{aligned}
\end{equation}
For the sake of simplicity, we will omit these subscripts and write $g$ instead of $g_{f,A,B,C}$ in the sequel, provided no confusion arises.\\
We denote the foliation $\mathcal{F}_1:=\mathcal{F}_f|_U$ and its orthogonal foliation (with respect to the metric $g$) by $\mathcal{F}_2$, i.e. the foliation generated by the distribution $\mathbb{R}E_3$, which is integrable since it is one-dimensional. 
\begin{remark}
    The vector fields $E_2$ and $E_3$ can only be defined on $\mathring{D}^2\setminus\{0\}\times S^1$ because $f$ is only defined for $r\in(0,1)$. But $E_1$ can be extended to $\partial(D^2\times S^1)\simeq T^2$.
\end{remark}
By expressing $(E_1,E_2,E_3)$ in terms of $(r,\phi,t)$, we obtain:
\begin{equation}
    \begin{aligned}
g&=g(E_1,E_1)E_1^{*2}+g(E_2,E_2)E_2^{*2}+g(E_3,E_3)E_3^{*2}\\
&=A(r)^2\text{d}\phi^2+B(r)^2(\frac{\text{d}r+f'(r)\text{d}t}{1+f'(r)^2})^2+C(r)^2(\frac{-f'(r)\text{d}r+\text{d}t}{1+f'(r)^2})^2\\
&= A^2\text{d}\phi^2+\frac{B^2+f'^2C^2}{(1+f'(r)^2)^2}\text{d}r^2+\frac{2f'(B^2-C^2)}{(1+f'^2)^2}\text{d}r\text{d}t+\frac{f'(r)^2B^2+C^2}{(1+f'(r)^2)^2}\text{d}t^2.
\end{aligned}
\end{equation}
Hence
\[
g=g_{rr}\text{d}r^2+g_{\phi\phi}\text{d}\phi^2+2g_{rt}\text{d}r\text{d}t+g_{tt}\text{d}t^2,
\]
where
\begin{equation}\label{g}
g_{\phi\phi}=A(r)^2, \quad g_{rr}=\frac{B^2+f'^2C^2}{(1+f'(r)^2)^2},\quad g_{rt}=\frac{f'(B^2-C^2)}{(1+f'^2)^2},\quad g_{tt}=\frac{f'(r)^2B^2+C^2}{(1+f'(r)^2)^2}.
\end{equation}
\subsection{The totally umbilical condition}
By Proposition \ref{prop fond}, the distribution $\text{T}\mathcal{F}_f$ (with its induced foliation denoted by $\mathcal{F}_1$, which is just $\mathcal{F}_f$) and its orthogonal distribution $\text{T}\mathcal{F}_f^\bot$ (induced foliation $\mathcal{F}_2$) determine a conformal product structure if and only if the leaves of $\mathcal{F}_1$ and $\mathcal{F}_2$ are all totally umbilical.

The unit normal vector field of $\mathcal{F}_f|_U$ is $N:=\frac{E_3}{C(r)}$. Then the scalar second fundamental form with respect to $N$ can be computed as follows for arbitrary vector fields $X,Y$ tangent to $\mathcal{F}_f$ 
\begin{equation}\label{II}
I\!I(X,Y):=g(h(X,Y),N)=g(\nabla_XY,N)=\frac{1}{C}g(\nabla_XY,E_3).
\end{equation}
\begin{remark}\label{rq1}
    Since $\mathcal{F}_1$ is of codimension-1, for any $X,Y\in\Gamma(\text{T}\mathcal{F}_1)$, $h(X,Y)=(\nabla_XY)^\bot\in\Gamma(T\mathcal{F}_1^\bot)=\text{span}\{N\}$. By definition, the 2-dimensional leaves of $\mathcal{F}_1$ are totally umbilical if $I\!I(X,Y) = \lambda g(X,Y), \forall X,Y\in\Gamma(\text{T}\mathcal{F}_1)$ for some scalar function $\lambda$. 
\end{remark}
We compute the Lie brackets:
\begin{equation}\label{lie brakets}
\begin{aligned}
[E_1,E_2]&=0,\ \ [E_1,E_3]=0\\
[E_2,E_3]&=(\partial_r+f'\partial_t)(-f'\partial_r+\partial_t)-(-f'\partial_r+\partial_t)(\partial_r+f'\partial_t)\\
&=-f^{\prime\prime}\partial_r+f'f^{\prime\prime}\partial_t\\
&=f^{\prime\prime}\frac{f'^2-1}{1+f'^2}E_2+\frac{2f'f^{\prime\prime}}{1+f'^2}E_3.
\end{aligned}
\end{equation}
We denote $[E_2,E_3]=\lambda E_2+\mu E_3$ with $\lambda=f^{\prime\prime}\frac{f'^2-1}{1+f'^2}$ and $\mu=\frac{2f'f^{\prime\prime}}{1+f'^2}$.
We look for functions $f,A,B,C$ such that the leaves of $\mathcal{F}_1$ and $\mathcal{F}_2$ are totally umbilical.

\begin{remark} 
The leaves of $\mathcal{F}_2$ are trivially totally umbilical because they are 1-dimensional. Since $\text{dim}T_2=1$, it is spanned by a single vector field $V$. The umbilical condition $h(X,Y)=g(X,Y)H$ is automatically satisfied for all $X,Y\in T_2$ by simply setting the mean curvature vector $H:=\frac{h(V,V)}{g(V,V)}$. Therefore, we only need to consider the umbilical condition for $\mathcal{F}_1$.
\end{remark}

\begin{proposition}\label{umbilical prop}
    The foliation $\mathcal{F}_f$ is umbilical on $(\mathring{D}^2\setminus\{0\})\times S^1$ with respect to the metric $g$ defined above if and only if there is a constant $\kappa\in\mathbb{R}\setminus\{0\}$ such that 
    \begin{equation}\label{umbilical}
        B(r)=A(r)\frac{1+f'(r)^2}{\kappa f'(r)},\quad \forall r\in(0,1).
    \end{equation}
\end{proposition}
\begin{proof}
We recall the Koszul formula:
\[
2g(\nabla_XY,Z)=Xg(Y,Z)+Yg(X,Z)-Zg(X,Y)+g([X,Y],Z)-g([X,Z],Y)-g([Y,Z],X).
\]
Since $\{E_i\}$ are orthogonal, we set $X=E_1$ and $Y=E_2$ and vice versa, combining with \eqref{lie brakets}, we find that $g(\nabla_{E_1} E_2, E_3) =g(\nabla_{E_2}E_1,E_3)=0$. Hence the second fundamental form satisfies $I\!I(E_1,E_2)=I\!I(E_2,E_1)=0$. 
Then the distribution $\mathcal{F}_1$ is umbilical if and only if $\frac{I\!I(E_1,E_1)}{g(E_1,E_1)}=\frac{I\!I(E_2,E_2)}{g(E_2,E_2)}$, which means the principal curvatures of the leaves in $\mathcal{F}_1$ in the direction of $E_1$ and $E_2$ are equal. We denote $k_i=\frac{I\!I(E_i,E_i)}{g(E_i,E_i)}, i\in\{1,2\}$ in order to simplify.

\textbf{i)} $E_1$ direction: we set $X=Y=E_1$ and $Z=E_3$, by the definition of the moving frame and the above computation of Lie brackets \eqref{lie brakets}, we obtain from the Koszul formula:
\[
2g(\nabla_{E_1}E_1,E_3)=-E_3g(E_1,E_1)=-E_3(A^2).
\]
we compute the derivative of $A^2$ in the direction of $E_3$:
\[
E_3(A^2)=(-f'\partial_r+\partial_t)(A(r)^2)=-2f'AA'.
\]
So $2g(\nabla_{E_1}E_1,E_3)=-(-2f'AA')=2f'AA'$, and the principal curvature reads:
\begin{equation}\label{k1}
k_1=\frac{I\!I(E_1,E_1)}{A^2}=\frac{1}{A^2C}g(\nabla_{E_1}E_1,E_3)=\frac{A'f'}{AC}.
\end{equation}

\textbf{ii)} $E_2$ direction: we set $X=Y=E_2$ and $Z=E_3$ in the Koszul formula, and compute
\[
2g(\nabla_{E_2}E_2,E_3)=-E_3g(E_2,E_2)-2g([E_2,E_3],E_2)
\]
The first term is:
\[
    -E_3(B^2)=-(-f'\partial_r+\partial_t)(B^2)=2f'BB'.
\]
and the second term:
\begin{equation}
\begin{aligned}
-2g([E_2,E_3],E_2)&=-2g(\lambda E_2+\mu E_3,E_2)=-2\lambda B^2=2f^{\prime\prime}\frac{1-f'^2}{1+f'^2}B^2.
\end{aligned}
\end{equation}
So we obtain the principal curvature in the direction $E_2$:
\begin{equation}\label{k2}
k_2=\frac{I\!I(E_2,E_2)}{B^2}=\frac{1}{C}(\frac{B'}{B}f'+f^{\prime\prime}\frac{1-f'^2}{1+f'^2}).
\end{equation}
The leaves in $\mathcal{F}_1$ are totally umbilical if and only if $k_1=k_2$, which means:
\[
\frac{A'f'}{AC}=\frac{1}{C}(\frac{B'}{B}f'+f^{\prime\prime}\frac{1-f'^2}{1+f'^2}).
\]
Since $f'(r)\neq0$, this equation is equivalent to:
\[
\frac{A'}{A}=\frac{B'}{B}+\frac{f^{\prime\prime}}{f'}\frac{1-f'^2}{1+f'^2},
\]
which means 
\[
\frac{\text{d}}{\text{d}r}(\text{ln}A)=\frac{\text{d}}{\text{d}r}(\text{ln}B)+\frac{\text{d}}{\text{d}r}(\text{ln}\frac{f'}{1+f'^2}).
\]
By integration, this equation holds if and only if there is a constant $\kappa$ such that
\begin{equation}\label{ode}
B(r)=A(r)\frac{1+f'(r)^2}{\kappa f'(r)}.
\end{equation}
This concludes the proof. 
\end{proof}
From now on we assume that $\mathcal{F}_f$ is totally umbilical on $\mathring{D}^2\setminus\{0\}\times S^1$, i.e., \eqref{ode} holds for $r\in(0,1)$.

By the construction, it is sufficient to check that the metric $g$ is smooth at the central circle $r=0$ and the torus foliation $r=1$.

\subsection{Smoothness condition at $r=0$}
In this section, we will find the condition for extending the metric $g$ associated with $f, A, B, C$ from $\mathring{D}^2\setminus\{0\}\times S^1$ to $\mathring{D}^2\times S^1$ smoothly. 
We consider the Euclidean coordinates $(x,y)$ of $D^2\subset\mathbb{R}^2$. The substitution relations are:
\begin{equation}\label{substitution}
\begin{cases}
    x=r\cos\phi,\\
    y=r\sin\phi,
\end{cases}\quad
\begin{cases}
    \text{d}x=\cos\phi\text{d}r-r \sin\phi\text{d}\phi,\\
    \text{d}y=\sin\phi\text{d}r+r \cos\phi\text{d}\phi.\\
\end{cases}
\end{equation}
By differentiating the equation $r^2=x^2+y^2$, we obtain $2r\text{d}r=2x\text{d}x+2y\text{d}y$, i.e., $r\text{d}r=x\text{d}x+y\text{d}y$. Hence
\[
\text{d}r^2=\frac{(x\text{d}x+y\text{d}y)^2}{x^2+y^2}=\frac{x^2}{x^2+y^2}\text{d}x^2+\frac{y^2}{x^2+y^2}\text{d}y^2+2\frac{xy}{x^2+y^2}\text{d}x\text{d}y.
\]

By \eqref{substitution}, we obtain: 
\begin{equation}
\begin{aligned}
x\text{d}y-y\text{d}x&=(r\cos\phi)(\sin\phi dr+r\cos\phi d\phi) - (r\sin\phi)(\cos\phi dr - r\sin\phi d\phi)\\
&= r\cos\phi\sin\phi dr + r^2\cos^2\phi d\phi - r\sin\phi\cos\phi dr + r^2\sin^2\phi d\phi\\
&=r^2(\cos^2\phi + \sin^2\phi) d\phi\\
&=r^2 d\phi
\end{aligned}
\end{equation}
Hence $\text{d}\phi=\frac{x}{x^2+y^2}\text{d}y-\frac{y}{x^2+y^2}\text{d}x$ and 
\[
\text{d}\phi^2=\left(\frac{-y\text{d}x+x\text{d}y}{x^2+y^2}\right)^2=\frac{y^2}{(x^2+y^2)^2}\text{d}x^2+\frac{x^2}{(x^2+y^2)^2}\text{d}y^2-2\frac{xy}{(x^2+y^2)^2}\text{d}x\text{d}y.
\]

\begin{equation}
\begin{aligned}
g=&g_{rr}\text{d}r^2+g_{\phi\phi}\text{d}\phi^2+2g_{rt}\text{d}r\text{d}t+g_{tt}\text{d}t^2\\
=&g_{rr}\left(\frac{x^2}{x^2+y^2}\text{d}x^2+\frac{y^2}{x^2+y^2}\text{d}y^2+2\frac{xy}{x^2+y^2}\text{d}x\text{d}y\right)\\
&+g_{\phi\phi}\left(\frac{y^2}{(x^2+y^2)^2}\text{d}x^2+\frac{x^2}{(x^2+y^2)^2}\text{d}y^2-2\frac{xy}{(x^2+y^2)^2}\text{d}x\text{d}y\right)\\
&+2g_{rt}\left(\frac{x\text{d}x+y\text{d}y}{\sqrt{x^2+y^2}}\right)\text{d}t+g_{tt}\text{d}t^2.
\end{aligned}
\end{equation}
So if we express the metric $g$ as:
\[
g=g_{xx}\text{d}x^2+g_{yy}\text{d}y^2+g_{tt}\text{d}t^2+2g_{xy}\text{d}x\text{d}y+2g_{xt}\text{d}x\text{d}t+2g_{yt}\text{d}y\text{d}t
\]
then we have:
\begin{equation}\label{terms}
\begin{aligned}
g_{xx}&=g_{rr}\frac{x^2}{x^2+y^2}+\frac{g_{\phi\phi}}{x^2+y^2}\frac{y^2}{x^2+y^2}=(g_{rr}- \frac{g_{\phi\phi}}{x^2+y^2}) \frac{x^2}{x^2+y^2} + \frac{g_{\phi\phi}}{x^2+y^2},\\
g_{yy}&=g_{rr}\frac{y^2}{x^2+y^2}+\frac{g_{\phi\phi}}{x^2+y^2}\frac{x^2}{x^2+y^2}=(g_{rr}-\frac{g_{\phi\phi}}{x^2+y^2})\frac{y^2}{x^2+y^2}+\frac{g_{\phi\phi}}{x^2+y^2},\\
g_{xy}&=(g_{rr}-\frac{g_{\phi\phi}}{x^2+y^2})\frac{xy}{x^2+y^2},\quad\\
g_{xt}&=g_{rt}\frac{x}{\sqrt{x^2+y^2}},\quad g_{yt}=g_{rt}\frac{y}{\sqrt{x^2+y^2}},\quad g_{tt}=g_{tt}.
\end{aligned}
\end{equation}
 The metric $g$ can be extended to $\mathring{D}^2\times S^1$ if and only if all the six terms above are smooth functions with variables $(x,y,t)$ defined on $\mathring{D^2}\times (-\pi,\pi)\subset\mathbb{R}^3$. The main ingredient we need is that the functions $\frac{g_{\phi\phi}}{x^2+y^2}$ , $(g_{rr}-\frac{{g}_{\phi\phi}}{x^2+y^2})\frac{1}{x^2+y^2}$, $\frac{g_{rt}}{\sqrt{x^2+y^2}}$ and $g_{tt}$ can be written as the composition of smooth functions in $\mathcal{C}^\infty((-1,1))$ and  $r^2(x,y,t)=x^2+y^2$, which is also smooth. Thus they are smooth functions defined on $\mathring{D}^2\times S^1$. We claim that the conditions of smoothness at $r=0$ of $g$ can be formalized as follows:
\begin{proposition}\label{condition at r=0}
    We assume that the metric $g$ associated with $f, A, B, C$ is a metric on $\mathring{D^2}\setminus\{0\}\times S^1$ such that the foliation $\mathcal{F}_f$ is umbilical in $\mathring{D^2}\setminus\{0\}\times S^1$, i.e., \eqref{umbilical} holds. Then $g$ can be extended smoothly at the central slice $C=\{0\}\times S^1$ if the following holds:\\
    i) There exists a smooth function $P\in\mathcal{C}^\infty((-1,1))$, such that $P(0)\neq0$ and $A(r)=r\cdot P(r^2)$.\\
    ii) $f^{\prime\prime}(0)=\frac1\kappa$.\\
    iii) There is a smooth positive function $H\in\mathcal{C}^\infty((-1,1))$ such that $C(r)=H(r^2)$.
\end{proposition}
\begin{proof}
We assume that i) ii) and iii) hold, and prove that all the terms in \eqref{terms} are smooth. 
\begin{lemma}
    There is a smooth function $\Delta\in\mathcal{C}^\infty((-1,1))$ such that $\Delta(r^2)=g_{rr}(r)-\frac{g_{\phi\phi}(r)}{r^2}$ for $r\in(0,1)$.
\end{lemma}
\begin{proof}
    We extend the function $f$ as in the proof of Lemma \ref{lemma again}. Then by Lemma \ref{lemma again} we know $f(r)=h(r^2)$ for a smooth function $h\in\mathcal{C}^\infty((-1,1))$, $f'(r)=h'(r^2)2r$. We set $E(x):=2h'(x)$ for $x\in(-1,1)$, hence $f'(r)=r\cdot E(r^2)$ for $r\in[0,1)$. By \eqref{g} and \eqref{ode}, we know 
    \begin{equation}
    \begin{aligned}
    g_{rr}(r)&=\frac{B(r)^2+f'(r)^2C(r)^2}{(1+f'(r)^2)^2}\\
    &=\frac{1}{(1+f'(r)^2)^2}\left[A(r)\frac{1+f'(r)^2}{\kappa f'(r)}\right]^2+\frac{f'(r)^2}{(1+f'(r)^2)^2}C(r)^2\\
    &=\frac{1}{(1+r^2E(r^2)^2)^2}\left[P(r^2)\frac{1+r^2E(r^2)^2}{\kappa E(r^2)}\right]^2+\frac{r^2E(r^2)^2}{(1+r^2E(r^2)^2)^2}H(r^2)^2\\
    &=\frac{P(r^2)^2}{\kappa^2E(r^2)^2}+\frac{r^2E(r^2)^2}{(1+r^2E(r^2)^2)^2}H(r^2)^2
    \end{aligned}
    \end{equation}
    Since $f^{\prime\prime}(r) = E(r^2) + 2r^2 E'(r^2)$, with condition ii), we obtain $E(0)=\frac1\kappa\neq0$. So if we set
    \[
    \Delta(x)=\frac{P(x)^2}{\kappa^2E(x)^2}+\frac{xE(x)^2}{(1+xE(x)^2)^2}H(x)^2-P(x)^2,
    \]
    then $\Delta$ is a smooth function in $\mathcal{C}^\infty((-1,1))$ and satisfies $\Delta(r^2)=g_{rr}-\frac{g_{\phi\phi}}{r^2}$ for any $r\in(0,1)$.
\end{proof}
Since $\Delta(0)=\frac{P(0)^2}{\kappa^2E(0)^2}-P(0)^2=0$, there is a smooth function $S\in\mathcal{C}^\infty((-1,1))$ such that $\Delta(x)=x\cdot S(x)$ for any $x\in(-1,1)$. The terms of $g$ in \eqref{terms} can be written as:
\begin{equation}\label{g xyz}
\begin{aligned}
g_{xx}&=\frac{\Delta(r^2)}{r^2}x^2+P(r^2)^2=S(r^2)x^2+P(r^2)^2,\\ g_{yy}&=\frac{\Delta(r^2)}{r^2}y^2+P(r^2)^2=S(r^2)y^2+P(r^2)^2,\\ g_{xy}&=\frac{\Delta(r^2)}{r^2}xy=S(r^2)xy.
\end{aligned}
\end{equation}
We now deal with the terms $g_{xt}$ and $g_{yt}$, by \eqref{g} we have
\[
\frac{g_{rt}(r)}{r}=\frac{f(r)'(B(r)^2-C(r)^2)}{r(1+f'(r)^2)^2}=\frac{A(r)^2}{\kappa^2r^2E(r^2)}-E(r^2)\frac{H(r^2)^2}{(1+r^2E(r^2)^2)^2}.
\]
Hence if we set $Q(x):=\frac{P(x)^2}{\kappa^2E(x)}-E(x)\frac{H(x)^2}{(1+xE(x)^2)^2}$ for any $x\in(-1,1)$, then $Q\in\mathcal{C}^\infty((-1,1))$ and $\frac{g_{rt}}{r}=Q(r^2)$ for all $r\in(0,1)$. So
\[
g_{xt}(x,y,t)=\frac{g_{rt}}{r}x=Q(x^2+y^2)x,\quad g_{yt}(x,y,t)=\frac{g_{rt}}{r}y=Q(x^2+y^2)y,
\]
hence they are smooth. 
The only remaining term $g_{tt}=g_{tt}(r)$ is smooth since by \eqref{g} and \eqref{ode}, we obtain:
\[
g_{tt}(r)=\frac{f'(r)^2B(r)^2+C(r)^2}{(1+f'(r)^2)^2}=\frac{r^2P(r^2)^2}{\kappa^2}+\frac{H(r^2)^2}{(1+r^2E(r^2)^2)^2}.
\]
This is a smooth function of variables $(x,y,t)$ since $r^2(x,y,t)=x^2+y^2$. This concludes the proof of smoothness of $g$. Note that $g$ is positive definite at $C=\{0\}\times S^1$ because $g_{xx}(0) = P(0)^2 > 0$  ($P(0)\neq 0$ by condition i)) $g_{yy}(0) = P(0)^2 > 0$ and $g_{tt}(0)=H(0)^2>0$. Hence $g$ extends to a well defined Riemannian metric.
\end{proof}
\begin{assumption}\label{assumption2}
We assume the functions $f,A,B,C$ satisfy:\\
i) $f$ satisfying the Assumptions \ref{assumption0} and \ref{assumption1},\\
ii) The umbilical equation \eqref{umbilical} holds for $r\in(0,1)$,\\
iii) Conditions i), ii), iii) in Proposition \ref{condition at r=0} hold.
\end{assumption}
\begin{proposition}\label{total umbilical}
If $f$, $A$, $B$, $C$ are functions such that Assumption \ref{assumption2} holds, then the foliation $\mathcal{F}_f$ is totally umbilical on $(\mathring{D^2}\times S^1, g)$, i.e., the umbilical property holds also at the central slice $C$.
\end{proposition}
\begin{proof}
By Proposition \ref{condition at r=0}, the metric $g$ and the foliation $\mathcal{F}_f$ extend smoothly to the central slice $C = \{0\} \times S^1$. Thus, the trace-free part of the second fundamental form, $\tilde{h} := h - \frac{1}{2}\text{tr}_g(h)g$, is a well-defined continuous tensor field on the entire solid torus $\mathring{D}^2\times S^1$. By Assumption \ref{assumption2} and Proposition \ref{umbilical prop}, $\mathcal{F}_f$ is totally umbilical for $r \in (0,1)$, meaning $\tilde{h} \equiv 0$ on the dense open subset $(\mathring{D}^2\setminus\{0\})\times S^1$. By continuity, $\tilde{h}$ must vanish everywhere. Hence, the leaves of $\mathcal{F}_f$ remain totally umbilical at $C$.
\end{proof}

\section{Conformal Product Structures on $S^3$}
We let $g_1:=\Psi_1^*g$ be the pullback of $g$ to $\mathring{V_1}$ through the diffeomorphism $\Psi_1$ from $\mathring{D}^2\times S^1$ to $\mathring{V}_1$, and similarly let $g_2:=\Psi_2^*g$ be the pullback of $g$ to $\mathring{V}_2$. We define the metric $\tilde{g}$ on $U=S^3\setminus (C_1\cup C_2\cup T^2)=\mathring{V_1}\cup\mathring{V}_2$ as: $\tilde{g}|_{\mathring{V_1}}=g_1$ and $\tilde{g}|_{\mathring{V_2}}=g_2$. In order to construct conformal product structures on $S^3$, what we need to do is to find the appropriate conditions on $f,A,B,C$ to ensure that: \\
i) the foliation $\tilde{\mathcal{F}_f}$ and the metric $g$ on $\mathring{D}^2\times S^1$ are smooth,\\
ii)  $\mathcal{F}_f$ is umbilical on $\mathring{D}^2\times S^1$,\\
iii) $\tilde{g}$ can be extended smoothly as a metric on $S^3$ (which we denote also as $\tilde{g}$) such that $\tilde{\mathcal{F}_f}$ is umbilical on $(S^3,\tilde{g})$.\\

We will work in the coordinate $(\rho,\alpha,\beta)$ of $O\subset S^3$ constructed in \eqref{V,O}. What we need to do is to find appropriate conditions to ensure the smoothness of $\tilde{g}$ on the torus $T^2$ i.e. smoothness of $\Phi^*\tilde{g}$ at the $\rho=0$ locus in the chart $(\Psi,W)$.
Let $g_{\mu\nu}^{(1)}$ and $g_{\mu\nu}^{(2)}$ denote the components of the pullback metrics $\Phi^*g_1$ and $\Phi^*g_2$ in the chart $W$ for $\rho > 0$ and $\rho < 0$ respectively. In order to prove that the extended metric $\tilde{g}$ is smooth at $T^2$ (i.e., $\Phi^*\tilde{g}$ can be extended smoothly at $\rho=0$ in $W$), we need to check that all its components match smoothly across the boundary, i.e.,
\begin{equation}\label{aim}
\lim_{\rho\rightarrow 0^+}\frac{d^n}{d\rho^n}g^{(1)}_{\mu\nu}(\rho)=\lim_{\rho\rightarrow0^-}\frac{d^n}{d\rho^n}g^{(2)}_{\mu\nu}(\rho),\quad\forall n\in\mathbb{N}, 
\end{equation}
for all $\mu,\nu\in\{\rho,\alpha,\beta\}$. 
\begin{proposition}\label{condition r=1}
    The metric $\tilde{g}$ can be extended smoothly at $T^2$ if the following holds in a neighborhood of $T^2$:\\
    i) There exists a smooth function $K$ defined on a neighborhood of $0$, such that $g_{rr}(r)=r^2\cdot K((1-r^2)^2)$ for $r$ sufficiently close to $1^-$.\\
    ii) The limit $\lim_{r\rightarrow1^-}A(r)$ exists and is strictly positive, and 
    \[\lim_{r\rightarrow1^-}\frac{\text{d}^n}{\text{d}r^n}A(r)=0, \quad\forall n\in\mathbb{N}_+
    \]
\end{proposition}
\begin{proof}

We assume that $\tilde{g}$ is smooth at $T^2$, which means that all of its terms in the coordinate $(\rho,\alpha,\beta)$ are smooth. We first compute $g_{\rho\rho}$ in $V_1$ and $V_2$ respectively:

If $\rho>0$ i.e., $(\rho,\alpha,\beta)\in V_1$, by \eqref{Psi_1 Phi}, we obtain the substitution relations:
\begin{equation*}
\begin{aligned}
&r\circ\Psi_1\circ\Phi(\rho,\alpha,\beta)=r(\sqrt{1-\rho}e^{i\beta},e^{i\alpha})=\sqrt{1-\rho},\\
&\phi\circ\Psi_1\circ\Phi(\rho,\alpha,\beta)=\phi(\sqrt{1-\rho}e^{i\beta},e^{i\alpha})=\beta,\\
&t\circ\Psi_1\circ\Phi(\rho,\alpha,\beta)=\text{Arg}(\gamma_1\circ\Psi_1\circ\Phi)=\alpha.
\end{aligned}
\end{equation*}
Hence we have 
\begin{equation*}
(\Psi_1\circ\Phi)^*\text{d}r=\text{d}(r\circ\Psi_1\circ\Phi)=\text{d}(\sqrt{1-\rho})=\frac{-1}{2\sqrt{1-\rho}}\text{d}\rho,
\end{equation*}
Similarly we obtain $(\Psi_1\circ\Phi)^*\text{d}\phi=\text{d}(\phi\circ\Psi_1\circ\Phi)=\text{d}\beta$, $(\Psi_1\circ\Phi)^*\text{d}t=\text{d}(t\circ\Psi_1\circ\Phi)=\text{d}\alpha$. Hence $(\Psi_1\circ\Phi)^*\text{d}r^2=\frac{1}{4(1-\rho)}\text{d}\rho^2$. By \eqref{g}, we obtain:
\begin{equation}\label{g_1}
\begin{aligned}
\Phi^*g_1&=(\Psi_1\circ\Phi)^*g=(\Psi_1\circ\Phi)^*\left(g_{rr}\text{d}r^2+g_{\phi\phi}\text{d}\phi^2+2g_{rt}\text{d}r\text{d}t+g_{tt}\text{d}t^2\right)\\
&=(g_{rr}\circ\Psi_1\circ\Phi)\left( \frac{-1}{2\sqrt{1-\rho}}d\rho\right)^2 +(g_{\phi\phi}\circ\Psi_1\circ\Phi)d\beta^2\\
&+2(g_{rt}\circ\Psi_1\circ\Phi)\left(\frac{-1}{2\sqrt{1-\rho}}\text{d}\rho\right)\text{d}\alpha+ (g_{tt}\circ\Psi_1\circ\Phi)\text{d}\alpha^2
\end{aligned}
\end{equation}
If we write $\Phi^*g_1$ as: $\Phi^*g_1=g_{\rho\rho}^{(1)}\text{d}\rho^2+g_{\alpha\alpha}^{(1)}\text{d}\alpha^2+g_{\beta\beta}^{(1)}\text{d}\beta^2+2g_{\rho\beta}^{(1)}\text{d}\rho\text{d}\beta+2g_{\rho\alpha}^{(1)}\text{d}\rho\text{d}\alpha+2g_{\alpha\beta}^{(1)}\text{d}\alpha\text{d}\beta$. Then by comparing with \eqref{g_1} we have $g_{\rho\beta}^{(1)}=g_{\alpha\beta}^{(1)}=0$ and
\begin{equation}\label{g_xx g_rr1}
g_{\rho\rho}^{(1)}(\rho,\alpha,\beta)=\frac{g_{rr}\circ\Psi_1\circ\Phi(\rho,\alpha,\beta)}{4(1-\rho)}=\frac{g_{rr}(\sqrt{1-\rho})}{4(1-\rho)}.
\end{equation}

For $\rho<0$ i.e., $(\rho,\alpha,\beta)\in V_2$, by \eqref{Psi_2 Phi}, we obtain the substitution relation of coordinates: $r\circ\Psi_2\circ\Phi(\rho,\alpha,\beta)=\sqrt{1+\rho}$, $\phi\circ\Psi_2\circ\Phi=\alpha$, $t\circ\Psi_2\circ\Phi=\text{Arg}(\gamma\circ\Psi_2\circ\Phi)=\beta$, we have 
\[
(\Psi_2\circ\Phi)^*\text{d}r=\frac{1}{2\sqrt{1+\rho}}\text{d}\rho,\quad(\Psi_2\circ\Phi)^*\text{d}\phi=\text{d}\alpha, \quad(\Psi_2\circ\Phi)^*\text{d}t=\text{d}\beta,
\]
hence $(\Psi_2\circ\Phi)^*\text{d}r_2^2=\frac{1}{4(1+\rho)}\text{d}\rho^2$. 
\begin{equation}\label{g_2}
\begin{aligned}
\Phi^*g_2=&(\Psi_{2}\circ\Phi)^*g=(\Psi_2\circ\Phi)^*\left(g_{rr}\text{d}r^2+g_{\phi\phi}\text{d}\phi^2+2g_{rt}\text{d}r\text{d}t+g_{tt}\text{d}t^2\right)\\
=&(g_{rr}\circ\Psi_2\circ\Phi)\left( \frac{1}{2\sqrt{1+\rho}}\text{d}\rho \right)^2+(g_{\phi\phi}\circ\Psi_2\circ\Phi)d\alpha^2\\
&+2(g_{rt}\circ\Psi_2\circ\Phi)\left(\frac{1}{2\sqrt{1+\rho}}\text{d}\rho \right)\text{d}\beta+ (g_{tt}\circ\Psi_2\circ\Phi)\text{d}\beta^2
\end{aligned}
\end{equation}
If we write the metric $\Phi^*g_2$ as: $\Phi^*g_2=g_{\rho\rho}^{(2)}\text{d}\rho^2+g_{\alpha\alpha}^{(2)}\text{d}\alpha^2+g_{\beta\beta}^{(2)}\text{d}\beta^2+2g_{\rho\alpha}^{(2)}\text{d}\rho\text{d}\alpha+2g_{\rho\beta}^{(2)}\text{d}\rho\text{d}\beta+2g_{\alpha\beta}^{(2)}\text{d}\alpha\text{d}\beta$.
Then by comparing with \eqref{g_2}, we have $g_{\rho\alpha}^{(2)}=g_{\alpha\beta}^{(2)}=0$ and 
\begin{equation}\label{g_xx g_rr2}
g_{\rho\rho}^{(2)}(\rho,\alpha,\beta)=\frac{g_{rr}\circ\Psi_2\circ\Phi(\rho,\alpha,\beta)}{4(1+\rho)}=\frac{g_{rr}(\sqrt{1+\rho})}{4(1+\rho)}.
\end{equation}
Combining \eqref{g_xx g_rr1} and \eqref{g_xx g_rr2}, we can define a unified function $g_{\rho\rho}(\rho)$ on $(-1,1)\setminus\{0\}$ representing the $\rho\rho$-component of the metric away from the boundary $T^2$:
\begin{equation}\label{g_xx g_rr}
   g_{\rho\rho}(\rho)=g_{rr}(\sqrt{1-|\rho|})\cdot\frac{1}{4(1-|\rho|)},\quad \rho\neq0
\end{equation}

\begin{lemma}\label{lem:5.2}
$g_{\rho\rho}(\rho)$ can be extended to a smooth function on the neighborhood $(-\epsilon, \epsilon)$ if $g_{rr}(r)$ can be expressed as:
\begin{equation}\label{g_rr}
g_{rr}(r)=r^2\cdot K((1-r^2)^2)
\end{equation}
for $r \in (\sqrt{1-\epsilon}, 1)$, where $K$ is a smooth function defined on $(-\epsilon^2, \epsilon^2)$. In this case, $g_{\rho\rho}(\rho)=\frac14 K(\rho^2)$.
\end{lemma}
\begin{proof}
For any $\rho\in(0,\epsilon)$, i.e., in the region corresponding to $\mathring{V}_1$, the radial coordinate function $r$ is pulled back via the transition map to the chart $W$ as $r \circ \Psi_1 \circ \Phi(\rho, \alpha, \beta) = \sqrt{1-\rho} \in (\sqrt{1-\epsilon}, 1)$.
By substituting \eqref{g_rr} into \eqref{g_xx g_rr1}, we obtain:
\[
g_{\rho\rho}^{(1)}(\rho) = \frac{g_{rr}(r \circ \Psi_1 \circ \Phi)}{4(1-\rho)} = \frac{g_{rr}(\sqrt{1-\rho})}{4(1-\rho)} = \frac{(\sqrt{1-\rho})^2 \cdot K\left((1-(\sqrt{1-\rho})^2)^2\right)}{4(1-\rho)} = \frac{1}{4} K(\rho^2).
\]
Similarly for $\rho\in(-\epsilon,0)$, corresponding to $\mathring{V}_2$, the radial coordinate is pulled back as $r \circ \Psi_2 \circ \Phi(\rho, \alpha, \beta) = \sqrt{1+\rho} \in (\sqrt{1-\epsilon}, 1)$.
By substituting into \eqref{g_xx g_rr2}, we have:
\[
g_{\rho\rho}^{(2)}(\rho) = \frac{g_{rr}(r \circ \Psi_2 \circ \Phi)}{4(1+\rho)} = \frac{g_{rr}(\sqrt{1+\rho})}{4(1+\rho)} = \frac{(1+\rho)\cdot K(\rho^2)}{4(1+\rho)} = \frac{1}{4} K(\rho^2).
\]
Thus, for all $\rho \in (-\epsilon, \epsilon) \setminus \{0\}$, the components $g_{\rho\rho}^{(1)}(\rho)$ and $g_{\rho\rho}^{(2)}(\rho)$ symmetrically coincide with the function $\frac14 K(\rho^2)$.
Since $K$ is smooth on $(-\epsilon^2, \epsilon^2)$ and the map $\rho \mapsto \rho^2$ is smooth, their composition $g_{\rho\rho}(\rho) = \frac14 K(\rho^2)$ is naturally smooth on the entire neighborhood $(-\epsilon, \epsilon)$.
\end{proof}

By \eqref{g}, we know $g_{rr}=\frac{B^2+f'^2C^2}{(1+f'^2)^2}$, which is equivalent to:
\begin{equation}\label{C(r)}
C(r)^2=\frac{(1+f'^2)^2}{f'^2}g_{rr}(r)-\frac{B(r)^2}{f'(r)^2}.
\end{equation}
If \eqref{g_rr} holds for a smooth function $K\in\mathcal{C}^\infty((-1,1))$, then we first express $g_{rt}$ and $g_{tt}$ by $g_{rr}$, hence by $K$.
With \eqref{ode} and \eqref{g}, we obtain:

\begin{equation}\label{g_rt}
\begin{aligned}
g_{rt}(r)&\overset{\eqref{g}}{=}\frac{f'(r)(B(r)^2-C(r)^2)}{(1+f'(r)^2)^2}\\
&\overset{\eqref{C(r)}}{=}\frac{f'(r)}{(1+f'(r)^2)^2}\left[B^2-\left(\frac{(1+f'(r)^2)^2}{f'(r)^2}g_{rr}-\frac{B^2}{f'(r)^2}\right)\right]\\
&=\frac{B^2}{f'(r)(1+f'(r)^2)}-\frac{g_{rr}}{f'(r)}\\
&\overset{\eqref{ode}}{=}\frac{1}{f'(r)}\left[\frac{A(r)^2(1+f'(r)^2)}{\kappa^2 f'(r)^2}-g_{rr}(r)\right]
\end{aligned}
\end{equation}
\begin{equation}\label{g_tt}
    \begin{aligned}
        g_{tt}(r)&\overset{\eqref{g}}{=}\frac{f'(r)^2B(r)^2+C(r)^2}{(1+f'(r)^2)^2}\\
        &\overset{\eqref{C(r)}}{=}\frac{1}{(1+f'(r)^2)^2}\left[f'(r)^2B(r)^2+\frac{(1+f'(r)^2)^2}{f'(r)^2}g_{rr}-\frac{B(r)^2}{f'(r)^2}\right]\\
        &\overset{\eqref{ode}}{=}\left( \frac{A(r)^2(1+f'(r)^2)^2}{\kappa^2 f'(r)^2} \right) \frac{f'(r)^2 - 1}{f'(r)^2(1+f'(r)^2)} + \frac{g_{rr}(r)}{f'(r)^2}\\
        &=\frac{A(r)^2}{\kappa^2} + \frac{1}{f'(r)^2} \left[ g_{rr}(r) - \frac{A(r)^2}{\kappa^2 f'(r)^2} \right]
    \end{aligned}
\end{equation}
We write the expression of $g_{tt}$ \eqref{g_tt} as: $g_{tt}(r)=\frac{A(r)^2}{\kappa^2}+R(r)$ where 
\begin{equation}
    R(r)=\frac{g_{rr}(r)}{f'(r)^2}-\frac{A(r)^2}{\kappa^2 f'(r)^4}.
\end{equation}

Now return to \eqref{g_1} and \eqref{g_2}, we can write $\Phi^*g_1$ and $\Phi^*g_2$ as:
\[
\Phi^*g_1=g_{\rho\rho}^{(1)}\text{d}\rho^2+g_{\alpha\alpha}^{(1)}\text{d}\alpha^2+g_{\beta\beta}^{(1)}\text{d}\beta^2+2g_{\rho\alpha}^{(1)}\text{d}\rho\text{d}\alpha
\]
where
\begin{equation}\label{main V_1}
    \begin{aligned}
    g_{\rho\rho}^{(1)}&=g_{rr}(\sqrt{1-\rho})\frac{1}{4(1-\rho)}=(1-\rho)K(\rho^2)\frac{1}{4(1-\rho)}=\frac14K(\rho^2)\\
        g_{\alpha\alpha}^{(1)}&=g_{tt}(\sqrt{1-\rho})=\frac{A(\sqrt{1-\rho})^2}{\kappa^2}+R(\sqrt{1-\rho})\\
        &=\frac{A(\sqrt{1-\rho})^2}{\kappa^2}+\frac{1}{f'(\sqrt{1-\rho})^2}[(1-\rho)K(\rho^2)-\frac{A(\sqrt{1-\rho})^2}{\kappa^2 f'(\sqrt{1-\rho})^2}]\\
        g_{\beta\beta}^{(1)}&=A(\sqrt{1-\rho})^2\\
       g_{\rho\alpha}^{(1)}&=\frac{-1}{2\sqrt{1-\rho}}g_{rt}(\sqrt{1-\rho})\\
        &=\frac{-1}{2\sqrt{1-\rho} \cdot f'(\sqrt{1-\rho})} \left[ \frac{A(\sqrt{1-\rho})^2 (1+f'^2)}{\kappa^2 f'^2}-(1-\rho) K(\rho^2)\right]
    \end{aligned}
\end{equation}
Similarly
\[
\Phi^*g_2=g_{\rho\rho}^{(2)}\text{d}\rho^2+g_{\alpha\alpha}^{(2)}\text{d}\alpha^2+g_{\beta\beta}^{(2)}\text{d}\beta^2+2g_{\rho\beta}^{(2)}\text{d}\rho\text{d}\beta
\]
where
\begin{equation}\label{main V_2}
    \begin{aligned}
        g_{\rho\rho}^{(2)}&=g_{rr}(\sqrt{1+\rho})\cdot\frac{1}{4(1+\rho)}=\left[ (1+\rho)K(\rho^2)\right]\frac{1}{4(1+\rho)} = \frac{1}{4}K(\rho^2)\\
        g_{\alpha\alpha}^{(2)}&= A(\sqrt{1+\rho})^2\\
        g_{\beta\beta}^{(2)}&=g_{tt}(\sqrt{1+\rho})=\frac{A(\sqrt{1+\rho})^2}{\kappa^2}+R(\sqrt{1+\rho})\\
&=\frac{A(\sqrt{1+\rho})^2}{\kappa^2} + \frac{1}{f'(\sqrt{1+\rho})^2} \left[(1+\rho)K(\rho^2) - \frac{A(\sqrt{1+\rho})^2}{\kappa^2 f'(\sqrt{1+\rho})^2} \right]\\
g_{\rho\beta}^{(2)}&= \frac{1}{2\sqrt{1+\rho}}g_{rt}(\sqrt{1+\rho})\\
&=\frac{1}{2\sqrt{1+\rho} \cdot f'(\sqrt{1+\rho})} \left[ \frac{A(\sqrt{1+\rho})^2 (1+f'^2)}{\kappa^2 f'^2}-(1+\rho) K(\rho^2)\right]
    \end{aligned}
\end{equation}

The metric $\tilde{g}$ can be extended to a smooth metric on $S^3$ if and only if $\Phi^*\tilde{g}$ can be extended to metric on $W$, i.e., for $\mu,\nu\in\{\rho,\alpha,\beta\}$, we have 
\[
\lim_{\rho\rightarrow0^+}\frac{\text{d}^n}{\text{d}\rho^n}g^{(1)}_{\mu\nu}(\rho)=\lim_{\rho\rightarrow0^-}\frac{\text{d}^n}{\text{d}\rho^n}g^{(2)}_{\mu\nu}(\rho).
\]
By Condition iv) of Assumption \ref{assumption1}, we can prove the following two lemma:
\begin{lemma}\label{lem}
\begin{equation}\label{R vanish}
\lim_{\rho\rightarrow0^\mp}\frac{\text{d}^n}{\text{d}\rho^n}R(\sqrt{1\pm\rho})=0,\quad \forall n\in\mathbb{N}
\end{equation}
\end{lemma} 
\begin{proof}
We first prove that:
    \[
    \lim_{r\rightarrow1^-}R^{(k)}(r)=0,\quad \forall k\in\mathbb{R} .
    \]
By Proposition \ref{condition2}, $\lim_{r\rightarrow1^-}\frac{\text{d}^k}{\text{d}r^k}\frac{1}{f'(r)}=0,\quad \forall k\in\mathbb{N}$, we can check that: 
\[
\lim_{r\rightarrow1^-}\frac{\text{d}^k}{\text{d}r^k}\frac{1}{f'(r)^2}=0,\quad \forall k\in\mathbb{N}.
\]
We know that $g_{rr}(r)$ is smooth in a neighborhood of $T^2$ by \eqref{g_rr}, and $A(r)^2$ is smooth by definition. Hence, $\lim_{r\rightarrow1^-}R^{(k)}(r)=0$.\\
Then by a simple induction we can prove \eqref{R vanish}.
\end{proof}
\begin{lemma}\label{5.4}
\begin{equation}\label{intersection}
\lim_{\rho \to 0^+} \frac{d^n}{d\rho^n} g_{\rho\alpha}^{(1)}(\rho) = 0,\quad \lim_{\rho\to0^-} \frac{d^n}{d\rho^n} g_{\rho\beta}^{(2)}(\rho)=0,\quad\forall n\in\mathbb{N}
\end{equation}
\end{lemma}
We omit this proof as it is similar to Lemma \ref{lem}.


In order for the metric to match continuously at $T^2$ ($\rho=0$), the limits of $g_{\alpha\alpha}$ from both sides must be equal:
\[
\lim_{\rho\rightarrow0^+} g_{\alpha\alpha}^{(1)}(\rho) = \lim_{\rho\rightarrow0^-} g_{\alpha\alpha}^{(2)}(\rho).
\]
By \eqref{main V_1} and \eqref{main V_2}, and using the fact that $\lim_{\rho\to 0} R = 0$ from Lemma \ref{lem}, this gives:
\[
\frac{A(1)^2}{\kappa^2} = A(1)^2.
\]
Since $A(1) > 0$ by condition ii), this uniquely forces the geometric constraint $\kappa^2 = 1$. Without loss of generality, we set $\kappa = 1$ from now on.

Assuming $\kappa = 1$, by \eqref{intersection} and \eqref{R vanish}, the higher-order smoothness of the metric at $T^2$ reduces to the matching of derivatives:
\begin{equation}\label{condition A}
\lim_{\rho\rightarrow0^+}\frac{\text{d}^n}{\text{d}\rho^n}A(\sqrt{1-\rho})^2=\lim_{\rho\rightarrow0^-}\frac{\text{d}^n}{\text{d}\rho^n}A(\sqrt{1+\rho})^2, \quad \forall n\in\mathbb{N},
\end{equation}
which is immediately satisfied by condition ii) of Proposition \ref{condition r=1}.
This concludes the proof.
\end{proof}

From now on we assume that the following holds:
\begin{assumption}\label{assumption3}
The foliation $\mathcal{F}_f$ and metric $g$ are defined by functions $A$, $B$, $C$, $f$ which satisfy Assumption \ref{assumption2}, such that conditions i), ii) in  Proposition \ref{condition r=1} also hold. 
\end{assumption}
Then the metric $\tilde{g}$ can be extended to a smooth metric on $S^3$, which we denote also $\tilde{g}$. Its pull back to $W$ is denoted by $\Phi^*\tilde{g}=g_{\rho\rho}\text{d}\rho^2+g_{\alpha\alpha}\text{d}\alpha^2+g_{\beta\beta}\text{d}\beta^2+2g_{\rho\alpha}\text{d}\rho\text{d}\alpha+2g_{\rho\beta}\text{d}\rho\text{d}\beta$, whence $g_{\mu\nu}$ is the unique smooth function obtained by extending $g_{\mu\nu}(\rho)=\begin{cases}
    g_{\mu\nu}^{(1)}(\rho), \rho>0\\
    g_{\mu\nu}^{(2)}(\rho),\rho<0
\end{cases}$. We will establish that $\tilde{\mathcal{F}_f}$ is umbilical on the whole $(S^3,\tilde{g})$.
\begin{remark}
    If the functions $f$, $A$, $B$, $C$ satisfy Assumption \ref{assumption3}, then $\tilde{\mathcal{F}}_f$ is umbilical on $(S^3,\tilde{g})$. By the exact same continuity argument used at the central slice, the boundary leaf $T^2$ is totally umbilical because the trace-free second fundamental form $\tilde{h}$ is a continuous tensor field on $S^3$ that vanishes identically on the dense open subset $S^3 \setminus T^2$.
\end{remark}

\section{Precise construction}
We will construct functions $f$, $A$, $B$, $C$ that satisfy Assumption \ref{assumption3}: \\
For example, if we set
\begin{equation}\label{A(r)}
A(r)=\frac{1}{I_0}\int_0^r\text{exp}(-\frac{1}{1-t^2})\text{d}t
\end{equation}
where $I_0=\int_0^1\text{exp}(-\frac{1}{1-t^2})\text{d}t$. Then $B(r)$ is determined by \eqref{ode} up to the constant $\kappa\neq0$. Recall from the matching condition at $T^2$ in Section 5 that we necessarily have $\kappa^2=1$. Without loss of generality, we take $\kappa=1$. And we set $f$ as in \eqref{f}. 

Then we choose a smooth function $\eta:[0,1]\rightarrow[0,1]$ such that: $\eta(u)=0$, $\forall u\in[0,\frac13]$ and $\eta(u)=1$, $\forall u\in[\frac23,1]$. We define:
\begin{equation}\label{g_rr(r)}
g_{rr}(r)=(1-\eta(r^2))A(r)^2\frac{1+f'(r)^2}{f'(r)^2}+\eta(r^2)r^2\cdot K((1-r^2)^2)
\end{equation}

We choose a positive constant $K_0$ that satisfies: $r^2K_0-\frac{A^2}{f'^2}>0$ for $r\ge\sqrt{\frac13}$. Let $K(x) \equiv K_0$ be the constant function. Then the term $\eta(r^2) \left( r^2 K_0 - \frac{A^2}{f'^2} \right)$ is positive for any $r\in(\sqrt\frac13,1)$. Then by \eqref{C(r)}, we have:
\begin{equation}
\begin{aligned}
    C(r)^2&\overset{\eqref{C(r)}}{=}\frac{(1+f'^2)^2}{f'^2}g_{rr}(r)-\frac{B(r)^2}{f'(r)^2}\\
    &\overset{\eqref{ode}}{=}\frac{(1+f'^2)^2}{f'^2}g_{rr}(r)-\frac{A^2(1+f'^2)^2}{f'^4}\\
    &\overset{\eqref{g_rr(r)}}{=}\frac{(1+f'^2)^2}{f'^2} \left[(1-\eta(r^2)) A^2 \frac{1+f'^2}{f'^2}+\eta(r^2)r^2K_0\right]-\frac{A^2(1+f'^2)^2}{f'^4}\\
    &=\frac{(1+f'^2)^2}{f'^2} \left[\eta(r^2)\left( r^2 K_0-\frac{A^2}{f'^2}\right)+A(r)^2(1-\eta(r^2))\right].
\end{aligned}
\end{equation}
Hence $C(r)^2$ is positive for any $r\in(0,1)$.

Then we need to prove that the Assumption \ref{assumption3} is satisfied:

\textbf{i)} We first establish the conditions in Proposition \ref{condition at r=0}.

Since the integrand $\text{exp}(-\frac{1}{1-t^2})$ is a smooth even function, $A(r)$ is a smooth odd function with $A(0)=0$. By Hadamard's lemma (Lemma \ref{hadamard}), there exists a smooth function $\tilde{P}$ such that $A(r) = r \tilde{P}(r)$. Because $A(r)$ is odd, $\tilde{P}(r)$ is necessarily an even smooth function. Then, by Lemma \ref{whitney}, there exists a smooth function $P$ such that $\tilde{P}(r) = P(r^2)$, which yields $A(r) = r P(r^2)$. Furthermore, $P(0) = \tilde{P}(0) = A'(0) =\frac{1}{I_0e} \neq 0$. Thus, condition i) in Proposition \ref{condition at r=0} is satisfied.

By the construction of $f$ \eqref{f}, we have $f'(r)=r\text{exp}(\frac{r^2}{1-r^2})$, it satisfies $f^{\prime\prime}(0)=1$. Since we have set $\kappa =1$, so $f^{\prime\prime}(0)=\frac{1}{\kappa}$. Thus the condition ii) in Proposition \ref{condition at r=0} has been satisfied. 

Now we consider the condition iii). For $r\le\sqrt{\frac13}$, we have $\eta(r^2)=0$, then $g_{rr}(r)=A^2\frac{1+f'^2}{f'^2}$, by substituting into \eqref{C(r)}, we obtain $C(r)^2=B(r)^2$. Hence smoothness. For $A(r)$ defined by \eqref{A(r)}, it satisfies $A(0)=0$.
With $\kappa=1$, Equation \eqref{ode} yields：
\[
C(r)=B(r)=A(r)\frac{1+f'(r)^2}{f'(r)}=P(r^2)\frac{1+r^2E(r^2)^2}{E(r^2)}, \quad\forall r\in(0,\sqrt\frac13).
\]
Thus $C(r)=H(r^2), \forall r\in(0,\sqrt\frac13)$ for the smooth function $H\in\mathcal{C}^\infty((-1,1))$ defined by $H(x):=P(x)\frac{1+xE(x)^2}{E(x)}$ with $H(0)=\frac{P(0)}{E(0)}>0$. The apparent singularity of $C(r)$ at $r=0$ caused by $f'(r)$ in the denominator is smoothly cancelled by the zero of $A(r)$ at the origin. Hence the condition iii) of Proposition \ref{condition at r=0} holds, thus $g$ is smooth at $C$. Hence $\tilde{g}$ is smooth at $C_1$ and $C_2$.

\textbf{ii)} Then we establish conditions in Proposition \ref{condition r=1}.\\
For $r\ge \sqrt{\frac23}$, we have $\eta(r^2)=1$. Hence $g_{rr}(r)=r^2\cdot K_0$, the term $g_{rr}$ satisfies the condition i) in Proposition \ref{condition r=1}.
\begin{lemma}
The function $A(r)$ in \eqref{A(r)} satisfies the condition ii) of Proposition \ref{condition r=1}.
\end{lemma}

\begin{proof}
First, by the definition of the normalization constant $I_0$, the limit evaluates to $\lim_{r\rightarrow 1^-}A(r) = \frac{1}{I_0}\int_0^1 \text{exp}(-\frac{1}{1-t^2})\text{d}t = 1$, which evidently exists and is strictly positive. 

Second, its first derivative $A'(r) = \frac{1}{I_0}\text{exp}(-\frac{1}{1-r^2})$ is a standard smooth flat function at $r=1$. For any integer $k \ge 1$, the higher derivative $A^{(k+1)}(r)$ is the product of $\text{exp}(-\frac{1}{1-r^2})$ and a rational function of $r$ whose poles are only at $r=\pm 1$. As $r \to 1^-$, the rapid exponential decay strongly dominates any polynomial growth of the rational factor, ensuring that $\lim_{r\to 1^-}A^{(n)}(r) = 0$ for all $n \ge 1$.
\end{proof}

Thus, conditions i) and ii) in Proposition \ref{condition r=1} are completely satisfied, and we have established the smooth extension at $T^2$.

\printbibliography
\end{document}